\documentclass[12pt,a4paper]{article}
\usepackage{amsmath}
\usepackage{amssymb}
\usepackage{amsthm}
\usepackage{graphicx}
\usepackage[colorlinks=true,allcolors=blue]{hyperref}
\newcommand{\mb}[1]{\mathbb{#1}}
\newcommand{\mf}[1]{\mathbf{#1}}
\newcommand{\Cay}{\operatorname{Cay}}
\newcommand{\oE}{\mf{\operatorname{E}}}
\newcommand{\G}{\mb{G}}
\newcommand{\R}{\mb{R}}
\newcommand{\Z}{\mb{Z}}
\newcommand{\Rep}{P}

\newcommand{\isomorphic}{\simeq}
\newcommand{\To}{\rightarrow}

\newtheorem{Lemma}{Lemma}

\newtheorem{Theorem}{Theorem}

\newtheorem{Conjecture}{Conjecture}

\newtheorem{Definition}{Definition}
\newtheorem{Question}{Question}

\newenvironment{proofof}[1]{\noindent\emph{Proof of #1.}}{\qed}

\title{Twin bent functions, strongly regular Cayley graphs, and Hurwitz-Radon theory}

\author{
Paul~Leopardi
\thanks{University of Melbourne.
\protect\url{mailto:paul.leopardi@gmail.com}}
}

\date{Resubmitted to JACODES Math: 16 April 2017}

\begin{document}

\maketitle

\begin{abstract}
The real monomial representations of Clifford algebras
give rise to two sequences of bent functions.
For each of these sequences, the corresponding Cayley graphs are
strongly regular graphs, and the corresponding sequences of strongly regular graph parameters coincide.
Even so, the corresponding graphs in the two sequences are not isomorphic, except in the first 3 cases.
The proof of this non-isomorphism is a simple consequence of a theorem of Radon.
\end{abstract}

\section{Introduction}
\label{sec-Introduction}
Two recent papers \cite{Leo14Constructions,Leo15Twin} describe and investigate two infinite sequences of bent functions and their Cayley graphs.
The bent function $\sigma_m$ on $\Z_2^{2 m}$ is described in the first paper \cite{Leo14Constructions}, on
generalizations of Williamson's construction for Hada\-mard matrices.
The bent function $\tau_m$ on $\Z_2^{2 m}$ is described in the second paper \cite{Leo15Twin},
which investigates some of the properties of the two sequences of bent functions.
In this second paper it is shown that the bent functions $\sigma_m$ and $\tau_m$ both correspond to Hada\-mard difference sets with the same parameters
\begin{align*}
(v_m,k_m,\lambda_m,n_m) &= (4^m, 2^{2 m - 1} - 2^{m-1}, 2^{2 m - 2} - 2^{m-1}, 2^{2 m - 2}),
\end{align*}
and that their corresponding Cayley graphs are both strongly regular with the same parameters $(v_m,k_m,\lambda_m,\lambda_m)$.

The main result of the current paper is the following.
\begin{Theorem}\label{HR-non-imomorphic-theorem}
The Cayley graphs of the bent functions $\sigma_m$ and $\tau_m$ are isomorphic only when $m=1, 2,$ or $3.$
\end{Theorem}

The remainder of the paper is organized as follows.
Section \ref{sec-Background} outlines some of the background of this investigation.
Section \ref{sec-Preliminaries} includes further definitions used in the subsequent sections.
Section~\ref{sec-Results} proves the main result, and resolves the conjectures and the question raised by the previous papers.
Section~\ref{sec-Discussion} puts these results in context, and suggests future research.

\section{Background}\label{sec-Background}
A recent paper of the author \cite{Leo14Constructions} describes a generalization of
Williamson's construction for Hada\-mard matrices \cite{Wil44}
using the real monomial representation of the basis elements of the Clifford algebras $\R_{m,m}$.

Briefly, the general construction uses some
\begin{align*}
\quad &A_k \in \{-1,0,1\}^{n \times n}, \quad B_k \in \{-1,1\}^{b \times b},
\quad k \in \{1,\ldots,n\},
\end{align*}
where the $A_k$ are \emph{monomial} matrices,
and constructs
\begin{align}
H &:= \sum_{k=1}^n A_k \otimes B_k,
\tag{H0}
\end{align}
such that
\begin{align}
H \in \{-1,1\}^{n b \times n b}
\quad
\text{and}
\quad
H H^T &= n b I_{(n b)},
\tag{H1}
\end{align}
i.e. $H$ is a Hada\-mard matrix of order $n b$.
The paper \cite{Leo14Constructions} focuses on a special case of the construction,
satisfying the conditions
\begin{align}
 A_j \ast A_k = 0 \quad (j \neq k)&, \quad \sum_{k=1}^n A_k \in \{-1,1\}^{n \times n},
\notag
\\
 A_k A_k^T &= I_{(n)},
\notag
\\
 A_j A_k^T + \lambda_{j,k} A_k A_j^T &= 0 \quad (j \neq k),
\label{constructions-4}
\\
 B_j B_k^T - \lambda_{j,k} B_k B_j^T &= 0 \quad (j \neq k),
\notag
\\
 \lambda_{j,k} &\in \{-1,1\},
\notag
\\
\sum_{k=1}^n  B_k B_k^T &= n b I_{(b)},
\notag
\end{align}
where $\ast$ is the Hada\-mard matrix product.

In Section 3 of the paper  \cite{Leo14Constructions},
it is noted that the Clifford algebra $\R^{2^m \times 2^m}$ has a canonical basis consisting of $4^m$ real monomial matrices,
corresponding to the basis of the algebra $\R_{m,m}$, with the following properties:

Pairs of basis matrices either commute or anticommute.
Basis matrices are either symmetric or skew,
and so the basis matrices $A_j, A_k$ satisfy
\begin{align}
 A_k A_k^T &= I_{(2^m)},
\quad
 A_j A_k^T + \lambda_{j,k} A_k A_j^T = 0 \quad (j \neq k),
\quad
\lambda_{j,k} \in \{-1,1\}.
\label{A-property-1}
\end{align}

Additionally, for $n=2^m$, we can choose a transversal of $n$ canonical basis matrices that
satisfies conditions \eqref{constructions-4} on the $A$ matrices,
\begin{align}
 A_j \ast A_k = 0 \quad (j \neq k)&,
\quad
\sum_{k=1}^n A_k \in \{-1,1\}^{n \times n}.
\label{A-property-2}
\end{align}

Section 3 also contains the definition of $\varDelta_m$, the restricted amicability / anti-amicability graph of $\R_{m,m}$,
and the subgraphs $\varDelta_m[-1]$ and $\varDelta_m[1]$, as well as the term ``transversal graph''.
These definitions are repeated here since they are used in the conjectures and question below.
\begin{Definition}\label{definition-delta}
\cite[p. 225]{Leo14Constructions}

Let $\varDelta_m$ be the graph whose vertices are the $n^2=4^m$
positive signed basis matrices of the real representation
of the Clifford algebra $\R_{m,m}$,
with each edge having one of two labels, $-1$ or $1$:
\begin{itemize}
\item
Matrices $A_j$ and $A_k$ are connected by an edge labelled by $-1$ (``red'') if they have disjoint support and are anti-amicable,
that is, $A_j A_k^{-1}$ is skew.
\item
Matrices $A_j$ and $A_k$ are connected by an edge labelled by $1$ (``blue'') if they have disjoint support and are amicable,
that is, $A_j A_k^{-1}$ is symmetric.
\item
Otherwise there is no edge between $A_j$ and $A_k$.
\end{itemize}
The subgraph $\varDelta_m[-1]$ consists of the vertices of $\varDelta_m$ and all edges in $\varDelta_m$ labelled by $-1$.
Similarly, the subgraph $\varDelta_m[1]$ contains all of the edges of $\varDelta_m$ that are labelled by $1$.
\end{Definition}

A \emph{transversal graph} for the Clifford algebra $\R_{m,m}$
is any induced subgraph of $\varDelta_m$ that is a complete graph on $2^m$ vertices.
That is, each pair of vertices in the transversal graph represents a pair of matrices,
$A_j$ and $A_k$ with disjoint support.

The following three conjectures appear in Section 3 of the paper \cite{Leo14Constructions}:

\begin{Conjecture}\label{conjecture-1}
For all $m \geqslant 0$ there is a permutation $\pi$ of the set of $4^m$ canonical basis matrices,
that sends an amicable pair of basis matrices with disjoint support to an anti-amicable pair, and vice-versa.
\end{Conjecture}

\begin{Conjecture}\label{conjecture-2}
For all $m \geqslant 0,$
for the Clifford algebra $\R_{m,m},$ the subset of transversal graphs that are
not self-edge-colour complementary
can be arranged into a set of pairs of graphs with each member of the pair
being edge-colour complementary to the other member.
\end{Conjecture}

\begin{Conjecture}\label{conjecture-3}
For all $m \geqslant 0,$
for the Clifford algebra $\R_{m,m},$ if a graph $T$ exists amongst the transversal graphs,
then so does at least one graph with edge colours complementary to those of $T$.
\end{Conjecture}
Note that Conjecture \ref{conjecture-1} implies  Conjecture \ref{conjecture-2},
which in turn implies Conjecture \ref{conjecture-3}.

The significance of these conjectures can be seen in relation to the following result,
which is Part 1 of Theorem 10 of the paper \cite{Leo14Constructions}.

\begin{Lemma}\label{th-graph-image}
If $b$ is a power of 2, $b=2^m$, $m \geqslant 0$,
the amicability / anti-amicability graph $P_b$ of the matrices
$\{-1,1\}^{b \times b}$ contains a complete two-edge-coloured graph on $2 b^2$ vertices
with each vertex being a Hada\-mard matrix.
This graph is isomorphic to $\varGamma_{m,m}$, the amicability / anti-amicability graph of
the group $\G_{m,m}$.
\end{Lemma}
The definitions of $\varGamma_{m,m}$ and $\G_{m,m}$
are given in Section 3 of the paper \cite{Leo14Constructions},
and the definition of $\G_{m,m}$ is repeated below.
For the current paper, it suffices to note that $\varDelta_m$ is a subgraph of $\varGamma_{m,m}$,
and so, therefore, are all of the transversal graphs.

An $n$-tuple of $A$ matrices of order $n=2^m$ satisfying properties \eqref{A-property-1} and \eqref{A-property-2}
yields a corresponding transversal graph $T$.
As noted in Section 5 of the paper  \cite{Leo14Constructions},
if Conjecture \ref{conjecture-3} were true, this would guarantee the existence
of an edge-colour complementary transversal graph $\overline{T}$.
In turn, because Lemma \ref{th-graph-image} guarantees the existence of a complete two-edge-coloured graph
isomorphic to $\varGamma_{m,m}$ within $P_b$, and because $\varDelta_m$ is a subgraph of $\varGamma_{m,m}$,
the graph $P_b$ would have to contain a two-edge-coloured subgraph isomorphic to $\overline{T}$.
This would imply the existence of an $n$-tuple of $B$ matrices of order $n$ satisfying the condition \eqref{constructions-4}
such that the construction (H0) would satisfy the Hada\-mard condition (H1), with a matrix of order $n^2$.

The author's subsequent paper on bent functions \cite{Leo15Twin}
refines Conjecture~\ref{conjecture-1} into the following question.
\begin{Question}
\label{Question-1}
Consider the sequence of edge-coloured graphs $\varDelta_m$ for $m \geqslant 1$,
each with red subgraph $\varDelta_m[-1],$ and blue subgraph $\varDelta_m[1].$
For which $m \geqslant 1$ is there an automorphism of $\varDelta_m$
that swaps the subgraphs $\varDelta_m[-1]$ and $\varDelta_m[1]$?
\end{Question}

The main result of this paper, Theorem \ref{HR-non-imomorphic-theorem},
leads to the resolution of these conjectures and this question.

\section{Further definitions and properties}
\label{sec-Preliminaries}
This section sets out the remainder of the definitions and properties used in this paper.
It is based on the previous papers \cite{Leo14Constructions, Leo15Twin} with additions.

\paragraph*{Clifford algebras and their real monomial representations.}
\label{sec-Clifford}

~

The following definitions and results appear in the paper on Hada\-mard matrices and Clifford algebras \cite{Leo14Constructions},
and are presented here for completeness, since they are used below.
Further details and proofs can be found in that paper, and in the paper on bent functions \cite{Leo15Twin},
unless otherwise noted.
An earlier paper on representations of Clifford algebras \cite{Leo05} contains more background material.

The signed group \cite{Cra95}
$\G_{p,q}$ of order $2^{1+p+q}$
is extension of $\Z_2$ by $\Z_2^{p+q}$,
defined by the signed group presentation
\begin{align*}
\G_{p,q} := \bigg\langle \
&\mf{e}_{\{k\}}\ (k \in S_{p,q})\ \mid
\\
&\mf{e}_{\{k\}}^2 = -1\ (k < 0), \quad \mf{e}_{\{k\}}^2 = 1\ (k > 0),
\\
&\mf{e}_{\{j\}}\mf{e}_{\{k\}} = -\mf{e}_{\{k\}}\mf{e}_{\{j\}}\ (j \neq k) \bigg\rangle,
\end{align*}
where $S_{p,q} := \{-q,\ldots,-1,1,\ldots,p\}.$
%

The $2 \times 2$ orthogonal matrices
\begin{align*}
\oE_1 :=
\left[
\begin{array}{cc}
0 & -1 \\
1 & 0
\end{array}
\right],
\quad
\oE_2 :=
\left[
\begin{array}{cc}
0 & 1 \\
1 & 0
\end{array}
\right]
\end{align*}
generate $\Rep(\G_{1,1}),$ the real monomial representation of group $\G_{1,1}.$
The cosets of $\{\pm I\} \equiv \Z_2$ in $\Rep(\G_{1,1})$ are
ordered using a pair of bits, as follows.
\begin{align*}
0 &\leftrightarrow 00 \leftrightarrow \{ \pm I \},
\\
1 &\leftrightarrow 01 \leftrightarrow \{ \pm \oE_1 \},
\\
2 &\leftrightarrow 10 \leftrightarrow \{ \pm \oE_2 \},
\\
3 &\leftrightarrow 11 \leftrightarrow \{ \pm \oE_1 \oE_2 \}.
\end{align*}

For $m > 1$,
the real monomial representation $\Rep(\G_{m,m})$ of the
group $\G_{m,m}$ consists of matrices of the form $G_1 \otimes G_{m-1}$
with $G_1$ in $\Rep(\G_{1,1})$ and $G_{m-1}$ in $\Rep(\G_{m-1,m-1}).$
The cosets of $\{\pm I\} \equiv \Z_2$ in $\Rep(\G_{m,m})$ are
ordered by concatenation of pairs of bits,
where each pair of bits uses the ordering as per $\Rep(\G_{1,1}),$
and the pairs are ordered as follows.
\begin{align*}
0 &\leftrightarrow 00 \ldots 00 \leftrightarrow \{ \pm I \},
\\
1 &\leftrightarrow 00 \ldots 01 \leftrightarrow \{ \pm I_{(2)}^{\otimes {(m-1)}} \otimes  \oE_1 \},
\\
2 &\leftrightarrow 00 \ldots 10 \leftrightarrow \{ \pm I_{(2)}^{\otimes {(m-1)}} \otimes  \oE_2 \},
\\
&\ldots
\\
2^{2m} - 1 &\leftrightarrow 11 \ldots 11 \leftrightarrow \{ \pm (\oE_1 \oE_2)^{\otimes {m}} \}.
\end{align*}
This ordering is called
the \emph{Kronecker product ordering} of the cosets of $\{\pm I\}$ in $\Rep(\G_{m,m}).$

The group $\G_{m,m}$ and its real monomial representation $\Rep(\G_{m,m})$
satisfy the following properties.
\begin{enumerate}
\item
Pairs of elements of $\G_{m,m}$ (and therefore $\Rep(\G_{m,m})$) either commute or anti\-commute:
for $g, h \in \G_{m,m},$ either $h g = g h$ or $h g = - g h.$
\item
The matrices $E \in \Rep(\G_{m,m})$ are orthogonal: $E E^T = E^T E = I.$
\item
The matrices $E \in \Rep(\G_{m,m})$ are either symmetric and square to give $I$ or
skew and square to give $-I$: either $E^T = E$ and $E^2 =I$ or $E^T = -E$ and $E^2 = -I.$
\end{enumerate}

Taking the positive signed element of each of the $2^{2m}$ cosets listed above
defines a transversal of $\{\pm I\}$ in $\Rep(\G_{m,m})$
which is also a monomial basis for the real representation of the Clifford algebra $\R_{m,m}$ in
Kronecker product order,
called this basis the \emph{positive signed basis} of $\Rep(\R_{m,m}).$

The function $\gamma_m : \Z_{2^{2 m}} \To \Rep(\G_{m,m})$
chooses the corresponding basis matrix from the positive signed basis of $\Rep(\R_{m,m}),$
using the Kronecker product ordering.
This ordering also defines a corresponding function on $\Z_2^{2 m},$
also called $\gamma_m.$

\paragraph*{Hurwitz-Radon theory.}
\label{sec-Hurwitz-Radon}
The key concept used in the proof of Lemma~\ref{Red-clique-lemma} below is that of a \emph{Hurwitz-Radon family} of matrices.
~

A set of real orthogonal matrices $\{A_1,A_2,\ldots,A_s\}$ is called a Hurwitz-Radon family
\cite{GerP74a,Hur22,Rad22} if
\begin{enumerate}
 \item
$A_j^T = -A_j$ for all $j=1,\ldots,s$, and
 \item
$A_j A_k = -A_k A_j$ for all $j \neq k$.
\end{enumerate}
The Hurwitz-Radon function $\rho$ is defined by
\begin{align*}
\rho(2^{4 d + c}) &:= 2^c + 8 d, \quad \text{where~} 0 \leqslant c < 4.
\end{align*}
As stated by Geramita and Pullman \cite{GerP74a}, Radon \cite{Rad22}
proved the following result, which is used as a lemma in this paper.
\begin{Lemma}\label{Hurwitz-Radon-lemma}
\cite[Theorem A]{GerP74a}

Any Hurwitz-Radon family of order $n$ has at most $\rho(n)-1$ members.
\end{Lemma}

\paragraph*{The two sequences of bent functions.}
\label{sec-Bent}

~

The previous two papers \cite{Leo14Constructions,Leo15Twin}
define two binary functions  on $\Z_2^{2 m}$, $\sigma_m$ and $\tau_m$, respectively.
Their key properties are repeated below.
See the two papers for the proofs and for more details and references on bent functions.

The function $\sigma_m : \Z_2^{2 m} \To \Z_2$ has the following properties.
\begin{enumerate}
 \item
For $i \in \Z_2^{2m},$ $\sigma_m(i) = 1$ if and only if the number of
digits equal to 1 in  the base 4 representation of $i$ is odd.
 \item
Since each matrix $\gamma_m(i)$ is orthogonal, $\sigma_m(i) = 1$ if and only if the matrix $\gamma_m(i)$ is skew.
 \item
The function $\sigma_m$ is bent.
\end{enumerate}

The function $\tau_m : \Z_2^{2 m} \To \Z_2$ has the following properties.
\begin{enumerate}
 \item
For $i \in \Z_2^{2m},$ $\tau_m(i) = 1$ if and only if the number of digits equal to 1 or 2 in the base 4
representation of $i$ is non zero, and the number of digits equal to 1 is even.
 \item
The value $\tau_m(i) = 1$ if and only if the matrix $\gamma_m(i)$ is symmetric but not diagonal.
 \item
The function $\tau_m$ is bent.
\end{enumerate}

\paragraph*{The relevant graphs.}
\label{sec-Graphs}

~

For a binary function $f : \Z_2^{2 m} \To \Z_2$, with $f(0)=0$ we consider the simple undirected \emph{Cayley graph} $\Cay(f)$  \cite[3.1]{BerC99}
where the vertex set $V(\Cay(f)) = \Z_2^{2 m}$ and for $i,j \in \Z_2^{2 m}$, the edge $(i,j)$ is in the edge set $E(\Cay(f))$ if and only if $f(i+j)=1$.

In the paper on Hada\-mard matrices \cite{Leo14Constructions} it is shown that
since $\sigma_m(i)=1$ if and only if $\gamma_m(i)$ is skew,
the subgraph $\varDelta_m[-1]$ is isomorphic to the Cayley graph $\Cay(\sigma_m)$.

The paper on bent functions \cite{Leo15Twin} notes that
since $\tau_m(i) = 1$ if and only if $\gamma_m(i)$ is symmetric but not diagonal,
the subgraph $\varDelta_m[1]$ is isomorphic to the Cayley graph $\Cay(\tau_m)$.
In that paper, these isomorphisms and the characterization of $\Cay(\sigma_m)$ and $\Cay(\tau_m)$
as Cayley graphs of bent functions are used to prove the following theorem.
\begin{Theorem}\label{Twins-are-strongly-regular-theorem}
\cite[Theorem 5.2]{Leo15Twin}

For all $m \geqslant 1,$
both graphs $\varDelta_m[-1]$ and $\varDelta_m[1]$ are strongly regular, with parameters
$v_m = 4^m,$ $k_m = 2^{2 m - 1} - 2^{m - 1},$ $\lambda_m=\mu_m=2^{2 m - 2} - 2^{m - 1}.$
\end{Theorem}

\section{Proof of Theorem \ref{HR-non-imomorphic-theorem} and related results}
\label{sec-Results}
Here we prove the main result, and examine its implications for Conjectures~\ref{conjecture-1} to~\ref{conjecture-3} and Question~\ref{Question-1}.

The proof of Theorem~\ref{HR-non-imomorphic-theorem} follows from the following two lemmas.
The first lemma puts an upper bound on the clique number of the graph $\Cay(\sigma_m) \isomorphic \varDelta_m[-1]$.
\begin{Lemma}
\label{Red-clique-lemma}
The clique number of the graph $\Cay(\sigma_m)$ is at most $\rho(2^m)$,
where $\rho$ is the Hurwitz-Radon function.
Therefore $\rho(2^m) < 2^m$ for $m \geqslant 4$.
\end{Lemma}
\begin{proof}
If we label the vertices of the graph $\Cay(\sigma_m)$ with the elements of $Z_2^{2m}$,
then any clique in this graph is mapped to another clique if a constant is added to all of the vertices.
Thus without loss of generality we can assume that we have a clique of order $s+1$ with one of the vertices labelled by 0.
If we then use $\gamma_m$ to label the vertices with elements of $\R_{m,m}$ to obtain the isomorphic graph $\varDelta_m[-1]$,
we have one vertex of the clique labelled with the identity matrix $I$ of order $2^m$.
Since the clique is in $\varDelta_m[-1]$, the other vertices $A_1$ to $A_s$ (say) must necessarily be skew matrices
that are pairwise anti-amicable,
\begin{align*}
A_j A_k^T &= -A_k A_j^T\quad\text{for all~} j \neq k.
\intertext{But then}
A_j A_k &= -A_k A_j\quad\text{for all~} j \neq k,
\end{align*}
and therefore $\{A_1,\ldots,A_s\}$ is a Hurwitz-Radon family.
By Lemma~\ref{Hurwitz-Radon-lemma}, $s$ is at most $\rho(2^m)-1$ and therefore the size of the clique is at most
$\rho(2^m)$.
\end{proof}

The second lemma puts a lower bound on the clique number of the graph $\Cay(\tau_m) \isomorphic \varDelta_m[1]$.
\begin{Lemma}
\label{Blue-clique-lemma}
The clique number of the graph $\Cay(\tau_m)$ is at least $2^m$.
\end{Lemma}
\begin{proof}
We construct a clique of order $2^m$ in $\Cay(\tau_m)$ with the vertices labelled in $\Z_2^{2m}$,
using the following set of vertices, denoted in base 4:
\begin{align*}
C_m &:= \{ 00 \ldots 00, 00 \ldots 02, 00 \ldots 20, \ldots, 22 \ldots 22 \}.
\end{align*}
The set $C_m$ is closed under addition in $\Z_2^{2 m}$,
and therefore forms a clique of order $2^m$ in $\Cay(\tau_m)$,
since the sum of any two distinct elements of $C_m$ is in the support of $\tau_m$.
\end{proof}

With these two lemmas in hand, the proof of Theorem~\ref{HR-non-imomorphic-theorem} follows easily.

\begin{proofof}{Theorem~\ref{HR-non-imomorphic-theorem}}
The result is a direct consequence of Lemmas~\ref{Red-clique-lemma} and~\ref{Blue-clique-lemma}.
For $m \geqslant 4$, the clique numbers of the graphs $\Cay(\sigma_m)$ and $\Cay(\tau_m)$ are different,
and therefore these graphs cannot be isomorphic.
\end{proofof}

Lemmas~\ref{Red-clique-lemma} and~\ref{Blue-clique-lemma}, along with Theorem~\ref{HR-non-imomorphic-theorem}
imply the failure of the conjectures~\ref{conjecture-1} to~\ref{conjecture-3}, as well as the resolution of Question~\ref{Question-1}, as follows.
\begin{Theorem}
\label{Conjectures-are-false-theorem}
For $m \geqslant 4$ the following hold.
\begin{enumerate}
 \item
There exist transversal graphs that do not have an edge-colour complement, and
therefore Conjecture~\ref{conjecture-3} does not hold.
\item
As a consequence, Conjectures~\ref{conjecture-1} and~\ref{conjecture-2} also do not hold.
\item
Question~\ref{Question-1} is resolved.
The only $m \geqslant 1$ for which there is an automorphism of $\varDelta_m$
that swaps the subgraphs $\varDelta_m[-1]$ and $\varDelta_m[1]$
are $m=1,2$ and $3$.
\end{enumerate}

\end{Theorem}

\begin{proof}
Assume that $m \geqslant 4$.
A transversal graph is a subgraph of $\varDelta_m$ which is a complete graph of order $2^m$.
The edges of a transversal graph are labelled with the colour red (if the edge is contained in $\varDelta_m[-1]$) or blue
(if the edge is contained in $\varDelta_m[1]$).
By Lemma~\ref{Red-clique-lemma}, the largest clique of $\varDelta_m[-1]$ is of order $\rho(2^m) < 2^m$,
and by Lemma~\ref{Blue-clique-lemma}, the largest clique of $\varDelta_m[1]$ is of order $2^m$.
If we take a blue clique of order $2^m$ as a transversal graph, this cannot have an edge-colour complement
in $\varDelta_m$, because no red clique can be this large.
More generally, we need only take a transversal graph containing a blue clique with order larger than $\rho(2^m)$ to have a clique with no edge-colour complement
in $\varDelta_m$.
This falsifies Conjecture~\ref{conjecture-3}.

Since Conjecture~\ref{conjecture-3} fails for $m \geqslant 4$,
the pairing of graphs described in Conjecture~\ref{conjecture-2} is impossible for $m \geqslant 4$.
Thus Conjecture~\ref{conjecture-2} is also false.

Finally, Conjecture~\ref{conjecture-1} fails as a direct consequence of Theorem~\ref{HR-non-imomorphic-theorem}
since, for $m \geqslant 4$, the subgraphs $\varDelta_m[-1]$ and $\varDelta_m[1]$are not isomorphic.
Therefore, for $m \geqslant 4$,  there can be no automorphism of $\varDelta_m$ that swaps these subgraphs.
\end{proof}

\section{Discussion}
\label{sec-Discussion}
The result of Lemma~\ref{Red-clique-lemma} is well known.
For example, the graph $\varDelta_m[-1]$ is the complement of the graph $V^+$ of Yiu \cite{Yiu90},
and the result for $V^+$ in his Theorem 2 is equivalent to Lemma \ref{Red-clique-lemma}.

The main consequence of Theorem~\ref{Conjectures-are-false-theorem} is that for $m>3$ there is at least one $n$-tuple of $A$ matrices,
with $n=2^m$ such that no $n$-tuple of $B$ matrices of order $n$ can be found to satisfy construction (H0) under condition (H1).
The proof of Theorem 5 of the Hada\-mard construction paper \cite{Leo14Constructions} shows by construction that for any $m$,
and any $n$-tuple of $A$ matrices satisfying \eqref{constructions-4}, there is an $n$-tuple of $B$ matrices of order $nc$ that
satisfies construction (H0) under condition (H1), where $c=M(n-1)$, with
\begin{align}
M(q) &:=
\begin{cases}
\lceil \frac{q}{2} \rceil + 1, \quad \text{if~} q \equiv 2,3,4 \quad (\operatorname{mod} 8),
\\
\lceil \frac{q}{2} \rceil \quad \text{otherwise.}
\end{cases}
\label{M-def}
\end{align}
Thus Theorem 5 remains valid.
The question remains as to whether the the order $nc$ is tight or can be reduced.
In the special case where the $n$-tuple of $A$ matrices is mutually amicable, the answer is given by Corollary 15 of the paper \cite{Leo14Constructions}:
The set of $\{-1,1\}$ matrices of order $c$ contains an $n$-tuple of mutually anti-amicable Hada\-mard matrices.
So in this special case, the required order can be reduced from $nc$ to $c$.
This leads to the following question.
\begin{Question}
In the general case, for any $m>1$, $n=2^m$, for any $n$-tuple of $A$ matrices satisfying \eqref{constructions-4},
does there always exist an $n$-tuple of $B$ matrices of order $c$ that
satisfies construction (H0) under condition (H1), where $c=M(n-1)$, with $M$ defined by \eqref{M-def}?
\end{Question}

As a result of Theorems~\ref{Twins-are-strongly-regular-theorem} and \ref{Conjectures-are-false-theorem},
we see that we have two sequences of strongly regular graphs, $\varDelta_m[-1]$ and $\varDelta_m[1]$ ($m \geqslant 1$),
sharing the same parameters,
$v_m = 4^m,$ $k_m = 2^{2 m - 1} - 2^{m - 1},$ $\lambda_m=\mu_m=2^{2 m - 2} - 2^{m - 1},$
but the graphs are isomorphic only for $m=1, 2, 3$.
For these three values of $m$, the existence of
automorphisms of $\varDelta_m$ that swap $\varDelta_m[-1]$ and $\varDelta_m[1]$
as subgraphs \cite[Table 1]{Leo14Constructions}
is remarkable in the light of Theorem~\ref{Conjectures-are-false-theorem}.

A paper of Bernasconi and Codenotti describes the relationship between bent functions and
their Cayley graphs, implying that a bent function corresponding to a $(v,k,\lambda,n)$ Hada\-mard difference set has a Cayley graph
that is strongly regular with parameters $(v,k,\lambda,\mu)$ where $\lambda=\mu$~\cite[Lemma 12]{BerC99}.
The current paper notes that for two specific sequences of bent functions, $\sigma_m$ and $\tau_m$,
the corresponding Cayley graphs are not necessarily isomorphic.

This raises the subject of classifying bent functions via their Cayley graphs, raising the following questions.
\begin{Question}
Which strongly regular graphs with parameters $(v,k,\lambda,\lambda)$ occur as Cayley graphs of bent functions?
\end{Question}
\begin{Question}
What is the relationship between other classifications of bent functions and the classification via Cayley graphs?
\end{Question}
This classification is the topic of a paper in preparation \cite{Leo16Classifying}.

With respect to the specific bent functions $\sigma_m$ and $\tau_m$ investigated here,
one of the anonymous reviewers of an earlier draft of this paper has asked whether each of these
functions are part of a larger class of bent functions.

The function $\sigma_m$ is a quadratic form, as can be seen from its definition and
its recursive identity \cite[Lemma 7]{Leo14Constructions}.
Specifically, $\sigma_m(0)=0$,
and, in terms of algebraic normal form,
using a particular convention for the mapping of bits to Boolean
variables, the identity is $\sigma_1(x_0,x_1)=x_0 x_1 + x_0$, and
\begin{align*}
\sigma_{m+1}(x_0,x_1,&\ldots,x_{2m},x_{2m+1})
=
\sigma_m(x_0,x_1) +
\sigma_m(x_2,x_3,\ldots,x_{2m},x_{2m+1})
\\
&= x_0 x_1 + x_0 + x_2 x_3 + x_2 + \ldots + x_{2m} x_{2m+1} + x_{2m}.
\end{align*}

In a paper in preparation \cite{Leo16Classifying},
it is proven that all quadratic bent functions with
the same dimension and weight have isomorphic Cayley graphs.

As for $\tau_m$, it is a \emph{bent iterative} function
\cite[Theorem~V.4]{CanCCP01cryptographic} \cite[Theorem~2]{CanC03decomposing} \cite{Tok11number},
as can be seen from its definition, and from the proof that it is a bent function
\cite[Theorem 3.1]{Leo15Twin}.

Since the $\mathcal{PS}^{(-)}$ \emph{partial spread}
bent functions are formed using $m$-di\-mensional subspaces of $\Z^{2m}$ which are disjoint except
for the $0$ vector \cite[p. 95]{Dil74},
these bent functions also have Cayley graphs whose clique number is at least $2^m$.
It could therefore be speculated that $\tau_m$ is also a $\mathcal{PS}^{(-)}$ bent function,
but exhaustive search using SageMathCloud \cite{SageMathCloud} shows that $\tau_3$ cannot be in
$\mathcal{PS}^{(-)}$.
Each clique of size 8 in $\Cay(\tau_3)$ that contains the 0 vector intersects each other such
clique at two vectors, only one of which is the 0 vector \cite{Leo16SMC}.

\paragraph*{Acknowledgements.}

Thanks to Christine Leopardi for her hospitality at Long Beach.
Thanks to Robert Craigen, Joanne Hall, William Martin,
Padraig {\'O} Cath{\'a}in and Judy-anne Osborn for valuable discussions.
This work was begun in 2014 while the author was a Visiting Fellow at the Australian National University,
continued while the author was a Visiting Fellow and a Casual Academic at the University of Newcastle, Australia,
and concluded while the author was an employee of the Bureau of Meteorology of the Australian
Government, and an Honorary Fellow of the University of Melbourne.
Thanks also to the anonymous reviewers of previous drafts of this paper.






\end{document}